\documentclass[11pt, oneside]{amsart} 

\usepackage[british,english]{babel} 
\usepackage{graphics,color,pgf}
\usepackage{epsfig}
\usepackage[ansinew]{inputenc}
\usepackage[all]{xy}
\usepackage{hyperref}
\newdir{ >}{!/8pt/@{}*@{>}}
\usepackage{amssymb, amsmath,amsthm, mathtools, amscd}
\usepackage{mathrsfs}
\usepackage{stmaryrd}
\usepackage[margin=1.5in]{geometry}

\makeatletter
\newtheorem*{rep@theorem}{\rep@title}
\newcommand{\newreptheorem}[2]{%
\newenvironment{rep#1}[1]{%
 \def\rep@title{#2 \ref{##1}}%
 \begin{rep@theorem}}%
 {\end{rep@theorem}}}
\makeatother

\theoremstyle{plain}
\newtheorem{teor}{Theorem}[section]
\newreptheorem{teor}{Theorem}  

\newtheorem{prop}[teor]{Proposition}
\newreptheorem{prop}{Proposition}  

\theoremstyle{definition}
\newtheorem{deft}[teor]{Definition}

\theoremstyle{remark}
\newtheorem{oss}[teor]{Remark}

\DeclareMathOperator\bbH{\mathbb{H}}

\DeclareMathOperator\bbP{\mathbb{P}}

\DeclareMathOperator\bbR{\mathbb{R}}

\DeclareMathOperator\calB{\mathcal{B}}

\DeclareMathOperator\calE{\mathcal{E}}

\DeclareMathOperator\calH{\mathcal{H}}

\DeclareMathOperator\calM{\mathcal{M}}

\DeclareMathOperator\calP{\mathcal{P}}

\DeclareMathOperator\scrF{\mathscr{F}}

\begin{document}

\title[Entropy rigidity for foliations by convex projective manifolds]{Entropy rigidity for foliations by strictly convex projective manifolds}

\author[A. Savini]{A. Savini}
\address{Section de Math\'ematiques, University of Geneva, Rue du Li\`evre 2, 1227 Geneva, Switzerland}
\email{Alessio.Savini@unige.ch}

\date{\today.\ \copyright{\ A. Savini 2020}. The author was partially supported by the FNS grant no. 200020-192216.}

\begin{abstract}

Let $N$ be a compact manifold with a foliation $\mathscr{F}_N$ whose leaves are compact strictly convex projective manifolds. Let $M$ be a compact manifold with a foliation $\mathscr{F}_M$ whose leaves are compact hyperbolic manifolds of dimension bigger than or equal to $3$. Suppose we have a foliation-preserving homeomorphism $f:(N,\mathscr{F}_N) \rightarrow (M,\mathscr{F}_M)$ which is $C^1$-regular when restricted to leaves. In the previous situation there exists a well-defined notion of foliated volume entropies $h(N,\scrF_N)$ and $h(M,\scrF_M)$ and it holds $h(M,\scrF_M) \leq h(N,\scrF_N)$. Additionally, if equality holds, then the leaves must be homothetic. 
 
\end{abstract}
  
\maketitle

\section{Introduction}

Natural maps revealed a very powerful tool in rigidity issues. For instance, inspired by the barycenter method introduced by Douady and Earle \cite{douady:earle}, Besson, Courtois and Gallot \cite{bcg95,bcg96,bcg98} defined natural maps to prove the minimal entropy conjecture for locally symmetric manifolds of rank one. Several applications of natural maps followed this striking result. For instance Francaviglia and Klaff \cite{franc06:articolo} and Francaviglia and the author \cite{savini:articolo,savini2:articolo} applied natural maps to study the $\textup{PO}(m,1)$-representation space of a real hyperbolic lattice $\Gamma \leq \textup{PO}(n,1)$, with $m \geq n \geq 3$. 

Other relevant applications of natural maps were obtained by Storm \cite{St06} in the proof of the minimal entropy conjecture for non uniform rank-one lattices. In the same spirit, Connell and Farb showed it for products of rank-one lattices \cite{connellfarb1}. They also extended the Degree Theorem for higher rank locally symmetric manifolds \cite{connellfarb2,connellfarb4}. A suitable adaptation of the previous results to the context of continuous foliations were studied by Boland and Connell \cite{boland:connell} and by Connell and Li \cite{connell:li}.

Recently natural maps have been exploited for studying measurable cocycles of hyperbolic lattices with values into rank-one Lie groups. The author \cite{savini:natural} introduced the notion of natural map associated to a measurable cocycle with an essentially unique map whose slices are atomless. The existence of such a map allows to defined the notion of natural volume and to state a rigidity result for maximal volume cocycles, in the spirit of Margulis \cite{margulis:ext} and Zimmer \cite{zimmer:annals}. 

In this paper we want to apply natural maps to the study of foliations by strictly convex projective manifolds. A manifold admits a strictly convex projective structure if it can be written as the quotient of a relatively compact domain in an affine chart of $\bbP^n(\mathbb{R})$ modulo a group of projective transformations acting freely and properly discontinuously. Thanks to the pioneering work by Benoist \cite{Ben00,Ben04,Ben06}, the interest in the study of convex projective structures grew rapidly. Inspired by the work done by Boland and Newberger \cite{boland:newberger} for Finsler manifolds, Adeboye, Bray and Constantine \cite{ABC,BC} proved an entropy rigidity result for strictly convex projective manifolds with finite volume. They showed that the volume entropy rescaled by the eccentricity factor associated to the strictly convex projective structure attains its minimum value if and only if the structure is actually hyperbolic. Notice that one has to suppose the existence of a hyperbolic structure a priori, since there exist strictly convex projective manifolds which do not admit any hyperbolic structure (see for instance Benoist \cite{Ben06} and Kapovich \cite{Kap07}).

Our purpose here is to study the same problem but in term of foliations. More precisely we are going to prove the following 

\begin{teor}\label{teor:foliation}
Let $(N,\mathscr{F}_N)$ be a compact manifold with a continuous foliation $\mathscr{F}_N$. Suppose that every leaf is a compact strictly convex projective manifold and that the structure varies continuously in the transverse direction. Let $\nu$ be a quasi-invariant transverse measure. Suppose we have $f:(N,\mathscr{F}_N) \rightarrow (M,\mathscr{F}_M)$ a foliation-preserving homeomorphism which is leafwise $C^1$-regular, with transversally continuous derivatives and such that $f_\ast \nu$-almost every leaf of $\mathscr{F}_M$ is a compact real hyperbolic manifold of dimension $n \geq 3$. Then there exist two measures $\mu_N$ and $\mu_M$ which can be written locally as the product of the transverse measures and the Busemann measures on the leaves and such that 
$$
\int_{M}h(L_M)^n d\mu_{M}(y) \leq \int_{N} h(L_N)^n \mathfrak{e}(L_N) d\mu_{N}(x) \ ,
$$ 
where $h( \ \cdot \ )$ and $\mathfrak{e}( \ \cdot \ )$ are the volume entropy and the eccentricity of the leaves. Additionally equality holds if and only if $\nu$-almost every leaf in $\mathscr{F}_N$ is homothetic to its image. 
\end{teor} 

Once can apply the previous theorem for instance to mapping tori associated to an automorphism of a compact strictly convex projective manifold where the projective structure of each fiber is deformed continuously. 

Notice that Theorem \ref{teor:foliation} is not a direct consequence of the main theorem of \cite{boland:connell}, since on the leaves of $\mathscr{F}_N$ we are not considering any Riemannian structure a priori. Additionally, it is worth mentioning that Theorem \ref{teor:foliation} holds also if we assume that the leaves of $N$ are Finsler manifolds instead of strictly convex projectives ones. In this way we obtain also a foliated version of \cite{boland:newberger}. 

The proof of the above theorem is quite easy since it relies on the construction of natural maps given by Adeboye, Bray and Constantine \cite{ABC}. Indeed it is sufficient to construct a natural map for every leaf of $\mathscr{F}_N$ and then glue all the natural maps together to get a global foliation-preserving map. Applying the foliated area formula by Boland and Connell \cite{boland:connell} one immediately concludes. 

\subsection*{Plan of the paper} 

In Section \ref{sec:intro:def} we introduce the background material we need throughout the paper. In Section \ref{sec:hilbert} we recall the notion of Hilbert geometry and we give the definition of strictly convex projective manifold. In Section \ref{sec:ps:theory} we recall the main aspects of the Patterson-Sullivan theory. The latter is necessary to give the definition of natural map. Section \ref{sec:teor} is devoted to the proof of the main theorem. We conclude with some comments about our main result. 

\subsection*{Acknowlegdements} I would like to thank the anonymous referee for her/his suggestions that allowed me to improve the quality of the paper. 

\section{Preliminary definitions}\label{sec:intro:def}

In this introductory section we are going to recall the main notions and results that we will need throughout the paper. We will first recall the definition of strictly convex proper domain $\Omega$ of $\mathbb{P}^n(\mathbb{R})$ and then we move on to the definition of Hilbert geometry. A strictly convex real projective manifold will be a quotient of $\Omega$ by a group $\Gamma$ of Hilbert isometries acting freely and properly discontinuously.

We are going also to recall briefly the work by Crampon \cite{crampon:tesi} about Patterson-Sullivan theory for Hilbert geometries. The existence of a Patterson-Sullivan density associated to a discrete group $\Gamma$ allowed Adeboye, Bray and Constantine \cite{ABC} to mimic the work by Besson, Curtois and Gallot \cite{bcg95,bcg96,bcg98} and to introduce the notion of natural map. We will conclude the section with a brief overview of this theory.

\subsection{Hilbert geometries}\label{sec:hilbert} 

In this section we will recall some elements about Hilbert geometries. Since the following will be a short exposition, we refer the reader for instance either to the Ph.D. thesis by Crampon \cite[Chapter 1]{crampon:tesi} or to the book by Busemann and Kelly \cite{buse:kelly} for more details.

Let $\Omega$ be a domain of $\mathbb{P}^n(\mathbb{R})$, that is an open connected set. We say that $\Omega$ is \emph{convex} if its intersection with any (projective) line is connected. This means that if $\Omega$ contains two points, then it must contain the segment determined by them. A convex domain $\Omega$ is \emph{proper} if it does not intersect some fixed projective hyperplane. Equivalently, $\Omega$ is projectively equivalent to a relatively compact domain in some affine chart. For a convex proper subset $\Omega$, we denote by $\partial \Omega$ its topological boundary. Any hyperplane intersecting $\partial \Omega$ but not $\Omega$ is called \emph{supporting hyperplane}. A proper domain $\Omega$ is \emph{strictly convex} if every supporting hyperplane intersects $\partial \Omega$ exactly in one point. 

Given any convex proper domain $\Omega$, we can define on it a distance as follows. Let $x,y \in \Omega$ be any two points and denote by $L_{xy}$ the line (or a line, if they coincide) determined by $x$ and $y$. If we denote by $a$ and $b$ the points on $\partial \Omega$ obtained by intersecting the boundary with $L_{xy}$, we can define 
$$
d_\Omega(x,y):=\left| \frac{1}{2} \log [a:x:y:b] \right|= \left| \frac{1}{2} \log \frac{ \lVert y-a \rVert \lVert x-b \rVert}{\lVert y - b \rVert \lVert x-a \rVert} \right| \ ,
$$
where $\lVert \ \cdot \ \rVert$ denotes the Euclidean norm in some affine chart where $\Omega$ is relatively compact. The distance $d_\Omega$ is called \emph{Hilbert distance} on $\Omega$.

It is worth noticing that any element $g \in \textup{PGL}(n+1,\bbR)$ which preserves $\Omega$ is automatically an isometry with respect to the Hilbert distance $d_\Omega$, since $g$ preserves the Euclidean cross-ratio. Additionally, in the particular case when $\Omega$ is strictly convex, projective transformations preserving $\Omega$ represent all the possible Hilbert isometries, as shown by de la Harpe \cite{Harpe93}. 

\begin{deft} \label{def:hilbert:geo}
Let $\Omega$ a convex proper domain of $\mathbb{P}^n(\mathbb{R})$. Let $d_\Omega$ the Hilbert distance and let $\textup{PGL}(\Omega)$ be the set of Hilbert isometries of $\Omega$. The triple $(\Omega,d_\Omega,\textup{PGL}(\Omega))$ is called \emph{Hilbert geometry}. 
\end{deft}

A well-known example of Hilbert geometry is given by the Klein model of the real hyperbolic space $\bbH^n$, where the strictly convex proper domain $\Omega$ is an ellipsoid. The latter case is the only one for which the Hilbert distance comes from a Riemannian metric defined on the domain. However, a Hilbert geometry naturally determines a Finslerian structure on the tangent space $T_p\Omega$ at any point $p \in \Omega$. More precisely, for any non-zero tangent vector $v \in T_p \Omega$ we define
$$
F_\Omega(p,v):=\frac{\lVert v \rVert}{2}\left( \frac{1}{\lVert p - p^+ \rVert} + \frac{1}{\lVert p - p^- \rVert} \right) \ ,
$$
where $p^\pm$ are the intersection points on $\partial \Omega$ determined by the line passing through $p$ with direction $v$. Here $\lVert \ \cdot \ \rVert$ is again the Euclidean norm in some affine chart. 

By defining the length of a $C^1$ curve $c:[0,1] \rightarrow \Omega$ as
$$
\ell(c):=\int_0^1 F_\Omega(c(t))dt \ ,
$$
it is easy to verify that the Hilbert distance $d_\Omega$ can be equivalently defined as
$$
d_{\Omega}(x,y)= \inf \{\ell(c)\ | \ \textup{$c$ is a $C^1$ curve}, \ c(0)=x, \ c(1)=y \} .
$$
The Hilbert distance will be said $C^k$-regular, if the Finslerian norm $F_\Omega:T \Omega \setminus \{0 \} \rightarrow \bbR$ is a $C^k$-regular function. Notice that the regularity of the latter is strictly related to the regularity of the boundary $\partial \Omega$. 

We are going to conclude this brief exposition about Hilbert geometries by the notion of divisible domain. 

\begin{deft}\label{def:sconvex:manifold}
Let $\Omega$ be a proper (strictly) convex domain of $\bbP^n(\bbR)$. A \emph{(strictly) convex real projective manifold} $Y$ is given by the quotient $Y=\Gamma \backslash \Omega$, where $\Gamma \leq \textup{PGL}(\Omega)$ is a group of projective transformations acting freely and properly discontinuously. We say that $\Omega$ is \emph{divisible} if $Y$ is compact. 
\end{deft}

An equivalent approach to convex real projective manifolds is given by the notion of geometric structure in the sense of Thurston \cite[Chapter 3]{thurston}. Indeed a manifold $Y$ is a convex real projective manifold if and only if it admits a $(\Omega,\textup{PGL}(\Omega))$-structure. Clearly real hyperbolic manifolds are particular examples of convex real projective manifolds. Other examples are given us by Coxeter groups studied by Kac and Vinberg \cite{KV67} and by deformations of hyperbolic lattices introduced by Johnson and Milson \cite{JM87}. For convex projective manifolds of finite volume, other examples have been recently obtained by Ballas and Marquis \cite{ballas:marquis} and by Ballas and Casella \cite{ballas:casella}. It is worth noticing that strictly convex projective structure are more general than hyperbolic ones. In fact there exist compact manifolds which admit strictly convex projective structures but no hyperbolic ones. Such examples are given for instance by Benoist \cite{Ben06} and Kapovich \cite{Kap07}.

In the case of compact manifolds, Benoist \cite{Ben04} proved that the condition of strict convexity of $\Omega$ is actually equivalent to requiring that the boundary $\partial \Omega$ is $C^1$-regular, and both conditions imply that the metric space $(\Omega,d_\Omega)$ is Gromov-hyperbolic. Similar results were later obtained also in the finite volume case by Cooper, Long and Tillman \cite[Theorem 0.15]{CLT15}. 

\subsection{Patterson-Sullivan theory for compact Hilbert manifolds}\label{sec:ps:theory} The following section is devoted to recall the main concepts about Patterson-Sullivan theory and their application to the definition of natural map introduced by Adeboye, Bray and Constantine \cite{ABC,BC}. We mainly refer to \cite{bray,crampon:tesi,ABC,BC}.

Let $\Omega$ be a strictly convex proper domain $\Omega$ of $\bbP^n(\bbR)$ and let $\Gamma$ be a group acting freely and properly discontinuously on $\Omega$ by projective transformations. We are going to suppose that the quotient $Y=\Gamma \backslash \Omega$ has finite volume  with respect to the Hilbert volume introduced in \cite{ABC}. As already noticed in Section \ref{sec:hilbert}, under the previous assumptions the geodesic boundary of $\Omega$ coincides with its topological boundary $\partial \Omega$, which is $C^1$-regular. Additionally any point of the boundary is \emph{smooth} (it has only one supporting hyperplane) and \emph{extremal} (since it is not contained in any line segment). Recall that for smooth extremal points in $\partial \Omega$ it is well-defined the notion of \emph{Busemann function} (whereas a priori there could be two Busemann functions as noticed by Bray \cite{bray}). More precisely, if we fix a basepoint $p \in \Omega$, we set
$$
\beta_p: \Omega \times \partial \Omega \rightarrow \bbR \ , \ \ \ \ \beta_p(x,\xi):=\lim_{t \to \infty} d_\Omega(x,c(t))-d_\Omega(p,c(t)) \ ,
$$
where $c$ is any path such that $c(0)=p$ and $c(\infty)=\xi$. The quantity $\beta_p(x,\xi)$ is nothing else than the signed distance between the point $x$ and the horosphere passing throgh $\xi$ and based at $p$, that is 
$$
\calH^\Omega_\xi(p):= \{ y \in \Omega \ | \ \beta_p(y,\xi)=0 \} \ .
$$

Busemann functions are a fundamental tool in order to introduce the notion of Patterson-Sullivan density, but before speaking about that, we need to give the following

\begin{deft}\label{def:critical:exponent}
Let $\Omega$ be a strictly convex proper domain and suppose that $\Gamma \leq \textup{PGL}(\Omega)$ acts freely and properly discontinuously with finite covolume. The \emph{critical exponent} $\delta_\Gamma$ associated to $\Gamma$ is defined as
$$
\delta_\Gamma:=\inf \{ s > 0  \ | \  \calP(x;s):=\sum_{\gamma \in \Gamma} e^{-sd_\Omega(x,\gamma(x))} < \infty \} \ , 
$$
for any point $x \in \Omega$. 
\end{deft}

One can prove that the definition is independent of the choice of the basepoint $x$. The function $\calP(x,s)$ is called \emph{Poincar\'e series} based at $x$. In the case that the Poincar\'e series diverges for $s=\delta_\Gamma$ we are going to say that $\Gamma$ is \emph{divergent}, otherwise it is called \emph{convergent}. For strictly convex domains, Crampon and Marquis \cite{Cra09,CM14} proved that $\delta_\Gamma$ is positive and it coincides with the volume growth entropy.

In the following definition we are going to denote by $\calM^1(\partial \Omega)$ the set of Borel probability measures on the boundary $\partial \Omega$.

\begin{deft}\label{def:patterson:sullivan}
Let $\alpha$ be a positive real number. An $\alpha$-\emph{dimensional Busemann density} is a measurable function 
$$
\mu:\Omega \rightarrow \calM^1(\partial \Omega), \ \ \mu(x)=\mu_x \ ,
$$
which satisfies the following properties:
\begin{enumerate}
	\item It is $\Gamma$-invariant, that is $\gamma_\ast \mu_x=\mu_{\gamma(x)}$, where $\gamma_\ast$ denotes the push-forward measure. 
	\item For every $x,y \in \Omega$ and $\xi \in \partial \Omega$, the measures $\mu_x$ and $\mu_y$ are absolutely continuous and it holds
$$
\frac{d\mu_y}{d\mu_x}(\xi)=e^{-\alpha \beta_x(y,\xi)} \ ,
$$
where $\beta_x(\cdot \ , \ \cdot)$ is the Busemann function pointed at $x$. 
\end{enumerate}
When $\alpha=\delta_\Gamma$ we are going to call $\mu$ \emph{Patterson-Sullivan density} associated to $\Gamma$.
\end{deft}

The generalization of Patterson-Sullivan theory to strictly-convex domains can be found for instance in \cite[Chapter 4.2]{crampon:tesi}. 

Patterson-Sullivan densities were used by Adeboye, Bray and Constantine \cite{ABC,BC} to introduce the notion of natural map in the spirit of Besson, Courtois and Gallot \cite{bcg95,bcg96,bcg98} and to prove a rigidity result for strictly convex real projective structures on a finite volume manifold admitting a hyperbolic structure. Before recalling the construction of the natural map, we need to fix some notation. We are going to denote by $Y=\Gamma \backslash \Omega$ a finite volume strictly convex real projective manifold and by $X=\Gamma \backslash \calE$ the same manifold with the hyperbolic structure coming from quotienting an ellipsoid $\calE$. Let $p \in \Omega$ and $o \in \calE$ be two base-points. We are going to use $\Omega$ and $\calE$ as superscripts to distinguish objects living in different spaces, such as Busemann functions. 

Given any Borel probability measure $\nu \in \calM^1(\partial \calE)$, one can define the following function
$$
\calB_\nu:\calE \rightarrow \bbR\ , \ \ \ \calB_\nu(x):=\int_{\partial \calE} \beta^{\calE}_o(x,\xi)d\nu(\xi) \ .
$$
Since Busemann functions are convex, the above function is strictly convex provided that the measure does not have any atom with weight greater than or equal to $1/2$. This implies that $\calB_\nu$ attains a unique minimum, as shown in \cite[Appendix A]{bcg95}. Thus in the previous situation one can define the \emph{barycenter of $\nu$} as
$$
\textup{bar}_{\calB}(\nu):=\textup{argmin}( \calB_\nu ) \ ,
$$
where the subscript $\calB$ refers to the dependence of the barycenter on the Busemann functions on $\calE$. Since both $Y$ and $X$ are actually the same manifold with different geometric structures, there exists a homeomorphism $f:Y \rightarrow X$ which can be lifted to the universal cover $\tilde f: \Omega \rightarrow \calE$. It is well-know that such a map $\tilde f$ induces a homeomorphism $\varphi:\partial \Omega \rightarrow \partial\calE$ between the geodesic boundaries of the universal covers (here we are exploiting the fact that $\Omega$ is a strictly convex domain). The boundary map $\varphi$ allows to define the \emph{natural map} as follows
$$
\tilde \Phi:\Omega \rightarrow \calE, \ \ \ \tilde\Phi(a):=\textup{bar}_{\calB}(\varphi_\ast(\mu_a)) \ ,
$$
where $\varphi_\ast(\mu_a)$ has no atoms since $\varphi$ is a homeomorphism. Hence we can apply the barycenter and the map $\Phi$ is well-defined. Additionally, since $\varphi$ is $\Gamma$-equivariant, the equivariance of the barycenter construction implies that $\tilde \Phi$ is $\Gamma$-equivariant as well and it induces a map $\Phi:Y \rightarrow X$. 

Using the implicit equation associated to the barycenter, Adeboye, Bray and Constantine \cite{ABC} proved that $\Phi$ is actually a smooth map whose Jacobian satisfies 
\begin{equation}\label{eq:jacobian:bound}
\textup{Jac}_a\Phi \leq \left( \frac{h(F_\Omega)}{h(g_0)} \right)^n \mathfrak{e}(F_\Omega) \ .
\end{equation}
for every $a \in \Omega$. Here $h(F_\Omega)$ is the volume entropy and $\mathfrak{e}(F_\Omega)$ is the eccentricity of the Finsler structure associated to the Hilbert geometry on $\Omega$. For more details about those numbers we refer to \cite{boland:newberger}. Additionally the equality in Equation \eqref{eq:jacobian:bound} is attained if and only if $D_a\Phi$ is a homothety between the tangent spaces. 

\section{Proof of the main theorem}\label{sec:teor}

In this section we are going to prove our main theorem. In order to do this we are going to fix once and for all the notation of the section. Let $(N,\mathscr{F}_N)$ be a compact manifold with a foliation $\mathscr{F}_N$ whose leaves are compact strictly convex projective manifolds. Similarly, let $(M,\mathscr{F}_M)$ be a compact manifold with a foliation $\mathscr{F}_M$ whose leaves are compact hyperbolic manifolds of dimension bigger than or equal to $3$. We are going to denote by $L_N$ (respectively $L_M$) a generic leaf of the foliation $\mathscr{F}_N$ (respectively $\mathscr{F}_M$). 

Suppose we have a foliation-preserving homeomorphism $f:(N,\mathscr{F}_N) \rightarrow (M,\mathscr{F}_M)$  whose restriction to leaves is $C^1$-regular, that is $f:L_N \rightarrow L_M$ is a $C^1$-map.

Following the line of \cite[Section 3]{boland:connell}, we are going to prove the existence of a foliation-preserving natural map. More precisely, we are going to prove the following

\begin{prop}\label{prop:natural:map}
There exists a measurable map $\Phi:(N,\mathscr{F}_N) \rightarrow (M,\mathscr{F}_M)$ which is foliation-preserving and it is $C^1$-regular when restricted to leaves. Additionally the leafwise Jacobian of such a restriction is uniformly bounded. 
\end{prop}

\begin{proof}
Since the restriction $f:L_N \rightarrow L_M$ to a generic leaf is a homeomorphism, following \cite{ABC,BC}, we have that its lift $\widetilde f: \widetilde{L}_N \rightarrow \widetilde{L}_M$ to the universal covers admits an extension to the visual boundaries $$\varphi:\partial \widetilde{L}_N \rightarrow \partial \widetilde{L}_M \ , $$ which is a homeomorphism. Thus we can define 
$$
\widetilde{\Phi}:\widetilde{L}_N \rightarrow \widetilde{L}_M \ , \ \ \ \widetilde{\Phi}(a):=\textup{bar}_\mathcal{B}(\varphi_\ast \mu_a) \ .
$$
where $\varphi_\ast \mu_a$ is the push-forward through $\varphi$ of the Patterson-Sullivan measure pointed at $a$. Notice that the map $\Phi$ is well-defined since $\varphi$ is a homeomorphism and hence $\varphi_\ast \mu_a$ has no atoms. Being $\varphi$ equivariant with respect to the action of the fundamental groups of $L_N$ and $L_M$ on the boundaries, the map $\widetilde{\Phi}$ is equivariant too, by the properties of the barycenter. Hence $\widetilde{\Phi}$ naturally descends to a map 
$$
\Phi:L_N \rightarrow L_M \ . 
$$
Being defined through the barycenter construction, the map $\Phi$ is $C^1$-regular with Jacobian bounded by 
\begin{equation}\label{eq:bound:jacobian:natural:leaf}
\textup{Jac}_a\Phi \leq \left(\frac{h(L_N)}{h(L_M)}\right)^n \mathfrak{e}(L_N) \ .
\end{equation}
where $h(L_N)$ (respectively $h(L_M)$) is the volume entropy associated to the Finsler structure on the leaf $L_N$ (respectively $L_M$), and $\mathfrak{e}(L_N)$ is the eccentricity. 

Notice that the same results contained in \cite[Section 5]{BC} show that $\Phi$ is (properly) homotopic to $f$ and hence $\Phi$ is surjective, having the same degree of $f$ (compare also with \cite[Theorem 3.3]{boland:connell} and \cite{boland05}).

Combining together all the maps we obtained, we get a map $$\overline{\Phi}:(N,\mathscr{F}_N) \rightarrow (M,\mathscr{F}_M) \ , $$ 
which is foliation-preserving and it is a bijection on the leaf spaces (since $f$ it is). Additionally $\overline{\Phi}$ is measurable by the transverse continuity of the metrics on the leaves. Indeed $\overline{\Phi}$ is defined in terms of the Busemann functions associated to the strictly convex structure given on each leaf. Since the Busemann functions varies measurably as the given projective structure varies continuosly and the same holds for the leafwise boundary extension $\varphi$ since we assumed $f$ transversally continuous, we get the desired measurability of $\overline{\Phi}$ (see also \cite{boland:connell}). 
\end{proof}

\begin{oss}\label{oss:integrability:jacobian:transverse:measure}
Since we are going to exploit it in the proof of the main theorem, it is worth noticing the Equation \eqref{eq:bound:jacobian:natural:leaf} implies that the laefwise Jacobian is actually uniformly bounded. Indeed the leafwise entropy $h( L_N )$ is uniformly bounded from above by $(n-1)$, where $n$ is the dimension of the leaf (see \cite[Theorem 1.1]{Cra09}). Additionally since $N$ is compact and the eccentricity varies continuously with respect to the leaves, then also the leafwise eccentricity $\mathfrak{e}(L_N)$ is uniformly bounded from above. 

This will guarantee that both the leawise Jacobian and the bound on the right appearing in Equation \eqref{eq:bound:jacobian:natural:leaf} are actually integrable functions with respect to the transverse quasi-invariant measure on $N$. 
\end{oss}

We need now to recall that by the work of Zimmer \cite[Theorem 2.1]{Zim82} on $(N,\mathscr{F}_N)$ there exists a \emph{transverse quasi-invariant measure} $\nu_N$. The notion of quasi-invariance refers to the holonomy associated to the foliation and for more details about the definition one can see for instance \cite{Zim82}. As noticed by Boland and Connell \cite{boland:connell} using the measure $\nu_N$ one can provide a global finite measure $\mu_N$ on $N$ which is locally the product of the Busemann measure on the leaves of $\mathscr{F}_N$ and of the measure $\nu_N$. 

One can exploit the natural map $\Phi$ constructed in Proposition \ref{prop:natural:map} to define in a similar way a measure on $M$. More precisely, set $$\nu_M=\Phi_\ast \nu_N \ ,$$
the push-forward measure of the transverse measure $\nu_N$ through the map $\Phi$. The measure $\nu_M$ is a transverse quasi-invariant measure and hence it can be used to define a global finite measure $\mu_M$ on $M$ which is locally the product of the Riemannian measure on the leaves of $\mathscr{F}_M$ and of the measure $\nu_M$. 

Now we have all we need to prove the main theorem.

\begin{proof}[Proof of Theorem \ref{teor:foliation}]
 
We are going to follow the line of \cite{boland:connell}. Thanks to the map $f:(N,\mathscr{F}_N) \rightarrow (M,\mathscr{F}_M)$ we know that we can define the natural map $$\Phi:(N,\mathscr{F}_N) \rightarrow (M,\mathscr{F}_M) \ , $$ and the global measures $\mu_N$ and $\mu_M$ described above. Let $(N_\alpha,\mathscr{F}_{N_\alpha},\nu_\alpha)$ be the ergodic decomposition described by Boland-Connell \cite{boland:connell} on $(N,\mathscr{F}_N)$ and let $(M_\alpha,\mathscr{F}_{M_\alpha},\Phi_\ast \nu_\alpha)$ be the corresponding decomposition on $(M,\mathscr{F}_M)$. Let $\mu_{N_\alpha}$ and $\mu_{M_\alpha}$ be the global measures associated to the ergodic components. 

Given $y \in M$, we denote by $\mathcal{N}(y)$ the cardinality of the preimage $\Phi^{-1}(y)$. It is worth noticing that the proof of foliated area formula \cite[Proposition 5.1]{boland:connell} holds mutatis mutandis in our context. As a consequence we get the following chain of inequalities 
\begin{align*}
\int_{M_\alpha}d\mu_{M_\alpha}(y) & \leq \int_{M_\alpha} \mathcal{N}(y)d\mu_{M_\alpha}(y) = \\
&= \int_{N_\alpha} \textup{Jac}_x\Phi \ d\mu_{N_\alpha}(x) \leq \\
& \leq \int_{N_\alpha} \left( \frac{h|_{N_\alpha}(L_N)}{h|_{M_\alpha}(L_M)} \right)^n\mathfrak{e}|_{N_\alpha}(L_N) d\mu_{N_\alpha}(x) \ ,
\end{align*}
where the first inequality is due to the surjectivity of $\Phi$ and we moved from the second line to the third one using the estimate on the Jacobian of $\Phi$. Notice that the last term is integrable as a consequence of Remark \ref{oss:integrability:jacobian:transverse:measure}. 

Since $h|_{M_\alpha}(L_M)$ is actually constant, we can rewrite
$$
\int_{M_\alpha}h(L_M)^n d\mu_{M_\alpha}(y) \leq \int_{N_\alpha} h|_{N_\alpha}(L_N)^n \mathfrak{e}|_{N_\alpha}(L_N) d\mu_{N_\alpha}(x) \ , 
$$
and integrating with respect to $\alpha$ we obtain
\begin{equation}\label{equation:rigidity}
\int_{M}h(L_M)^n d\mu_{M}(y) \leq \int_{N} h(L_N)^n \mathfrak{e}(L_N) d\mu_{N}(x) \ , 
\end{equation}
as desired. 

Assume now that equality holds in Equation \eqref{equation:rigidity}. For almost every $\alpha$ it must hold 
$$
\textup{Jac}_x(\Phi|_{N_\alpha}) = \left( \frac{h|_{N_\alpha}(L_N)}{h|_{M_\alpha}(L_M)} \right)^n\mathfrak{e}|_{N_\alpha}(L_N) \ ,
$$
almost everywhere with respect to $\mu_{N_\alpha}$. Thus we must have 
$$
\textup{Jac}_x(\Phi|_{N_\alpha}) = \left( \frac{h|_{N_\alpha}(L_N)}{h|_{M_\alpha}(L_M)} \right)^n\mathfrak{e}|_{N_\alpha}(L_N) 
$$
for $\nu_{N_\alpha}$-every leaf, and hence $D_x\Phi|_{N_\alpha}$ is a homotethy. This implies that $D_x\Phi$ is a homotethy $\nu_N$-almost every leaf, and the theorem is proved. 

\end{proof}

As already noticed in the introduction the previous theorem gives an adaptation of the main theorem of Boland and Connell \cite{boland:connell} in the case of foliation by strictly convex projective manifolds. Theorem \ref{teor:foliation} is not a direct consequence of such result because here we are not assuming any Riemannian structure on the leaf $L_N$ and hence there is no precise meaning of Patterson Sullivan structure in the sense of \cite{boland:connell}. Additionally, one can notice that following the same proof of Theorem \ref{teor:foliation} and using the constructions by Boland and Newberger \cite{boland:newberger} it is possible to replace the strictly convex projective structure on the leaves of $\mathscr{F}_N$ with a Finsler structure. 

\bibliographystyle{amsalpha}

\bibliography{biblionote}

\providecommand{\bysame}{\leavevmode\hbox to3em{\hrulefill}\thinspace}
\providecommand{\MR}{\relax\ifhmode\unskip\space\fi MR }
\providecommand{\MRhref}[2]{%
  \href{http://www.ams.org/mathscinet-getitem?mr=#1}{#2}
}
\providecommand{\href}[2]{#2}
\begin{thebibliography}{BCG98}

\bibitem[ABC19]{ABC}
I.~Adeboye, H.~Bray, and D.~Constantine, \emph{Entropy rigidity and hilbert
  volume}, Discrete Contin. Dyn. Syst. Ser. A \textbf{39} (2019), no.~4,
  1731--1744.

\bibitem[BCa]{ballas:casella}
S.~A. Ballas and A.~Casella, \emph{Gluing equations for real projective
  structures on $3$-manifolds}, preprint,
  \texttt{https://arxiv.org/pdf/1912.12508.pdf}.

\bibitem[BCb]{BC}
H.~Bray and D.~Constantine, \emph{Entropy rigidity for finite volume strictly
  convex projective manifolds}, preprint,
  \texttt{https://arxiv.org/pdf/2006.13619.pdf}.

\bibitem[BC02]{boland:connell}
J.~Boland and C.~Connell, \emph{Minimal entropy rigidity for foliations of
  compact spaces}, Israel J. Math. \textbf{128} (2002), 221--246.

\bibitem[BCG95]{bcg95}
G.~Besson, G.~Courtois, and S.~Gallot, \emph{Entropies et rigidit\'es des
  espaces localmentes sym\'etriques de courboure strictement n\'egative}, Geom.
  Func. Anal \textbf{5} (1995), no.~5, 731--799.

\bibitem[BCG96]{bcg96}
\bysame, \emph{Minimal entropy and mostow's rigidity theorems}, Ergodic Theory
  Dynam. Systems \textbf{16} (1996), 623--649.

\bibitem[BCG98]{bcg98}
\bysame, \emph{A real schwarz lemma and some applications}, Rend. Mat. Appl.
  \textbf{18} (1998), no.~2, 381--410.

\bibitem[BCS05]{boland05}
J.~Boland, C.~Connell, and J.~Souto, \emph{Volume rigidity for finite volume
  manifolds}, Amer. J. Math. \textbf{127} (2005), no.~3, 535--550.

\bibitem[Ben00]{Ben00}
Y.~Benoist, \emph{Automorphismes des cones convexes}, Invent. Math.
  \textbf{141} (2000), no.~1, 149--193.

\bibitem[Ben04]{Ben04}
\bysame, \emph{Convexes divisibles 1}, Tata Inst. Fund. Res. Stud. Math.
  \textbf{17} (2004), 339--374.

\bibitem[Ben06]{Ben06}
\bysame, \emph{Convexes hyperboliques et quasi-isometries}, Mat. Zametki
  \textbf{122} (2006), 109--134.

\bibitem[BK53]{buse:kelly}
Herbert Busemann and Paul~J. Kelly, \emph{Projective geometry and projective
  metrics}, Academic Press Inc, 1953.

\bibitem[BM]{ballas:marquis}
S.~A. Ballas and L.~Marquis, \emph{Properly convex bending of hyperbolic
  manifolds}, preprint, \texttt{https://arxiv.org/pdf/1609.03046.pdf}.

\bibitem[BN01]{boland:newberger}
J.~Boland and F.~Newberger, \emph{Minimal entropy rigidity for finsler
  manifolds of negative flag curvature}, Ergodic Theory Dynam. Systems
  \textbf{21} (2001), no.~1, 13--23.

\bibitem[Bra]{bray}
H.~Bray, \emph{Geodesic flow of nonstrictly convex hilbert geometries},
  preprint, https://arxiv.org/pdf/1710.06938.pdf.

\bibitem[CF03a]{connellfarb2}
C.~Connell and B.~Farb, \emph{The degree in higher rank}, J. Differential Geom.
  \textbf{65} (2003), no.~1, 19--59.

\bibitem[CF03b]{connellfarb1}
\bysame, \emph{Minimal entropy rigidity for lattices in products of rank one
  symmetric spaces}, Comm. Anal. Geom. \textbf{11} (2003), no.~5, 1001--1026.

\bibitem[CF17]{connellfarb4}
\bysame, \emph{Erratum for "the degree in higher rank"}, J. Differential Geom.
  \textbf{105} (2017), no.~1, 21--32.

\bibitem[CL]{connell:li}
C.~Connell and Z.~Li, \emph{Foliated entropy rigidity}, preprint.

\bibitem[CLT15]{CLT15}
D.~Cooper, D.~D. Long, and S.~Tillman, \emph{On convex projective manifolds and
  cusps}, Advances in Math. \textbf{277} (2015), 181--251.

\bibitem[CM14]{CM14}
M.~Crampon and L.~Marquis, \emph{Let flot g\'eod\'esique des quotients
  g\'eom\'etriquement finis des g\'eom\'etries de hilbert}, Pacific J. Math
  \textbf{268} (2014), no.~2, 313--369.

\bibitem[Cra09]{Cra09}
M.~Crampon, \emph{Entropies of strictly convex projective manifolds}, J. Mod.
  Dyn. \textbf{3} (2009), no.~4, 511--547.

\bibitem[Cra11]{crampon:tesi}
\bysame, \emph{Dynamics and entropies of hilbert metrics}, Ph.D. thesis,
  Universit\'e de Strasbourg, 2011.

\bibitem[DE86]{douady:earle}
A.~Douady and C.~J. Earle, \emph{Conformally natural extension of
  homeomorphisms of the circle}, Acta Math. \textbf{157} (1986), no.~1-2,
  23--48.

\bibitem[dlH93]{Harpe93}
P.~de~la Harpe, \emph{On hilbert's metric for simplices}, Geometric group
  theory, Vol. 1, London Math. Soc. Lecture Note Ser., vol. 181, Cambridge
  Univ. Press, Cambridge, 1993, pp.~97--119.

\bibitem[FK06]{franc06:articolo}
S.~Francaviglia and B.~Klaff, \emph{Maximal volume representations are
  fuchsian}, Geom.~Dedicata \textbf{117} (2006), 111--124.

\bibitem[FS18]{savini:articolo}
S.~Francaviglia and A.~Savini, \emph{Volume rigidity at ideal points of the
  character variety of hyperbolic 3-manifolds}, to appear on Ann. Sc. Norm.
  Super. Pisa Cl. Sci. (2018).

\bibitem[JM87]{JM87}
D.~Johnson and J.~J. Millson, \emph{Deformation spaces associated to compact
  hyperbolic manifolds}, Discrete groups in geometry and analysis, Progr.
  Math., vol.~67, Birkhauser Boston, 1987, pp.~48--106.

\bibitem[Kap07]{Kap07}
M.~Kapovich, \emph{Convex projective structures on gromov-thurston manifolds},
  Mat. Zametki \textbf{11} (2007), 1777--1830.

\bibitem[KV67]{KV67}
V.~G. Kac and E.~B. Vinberg, \emph{Quasi-homogeneous cones}, Mat. Zametki
  \textbf{1} (1967), 347--354.

\bibitem[Mar75]{margulis:ext}
G.~A. Margulis, \emph{Non-uniform lattices in semisimple algebraic groups}, Lie
  groups and their representations (Proc.~Summer School on Group
  Representations of the Bolyai J\'{a}nos Math. Soc., Budapest, 1971 (1975),
  371--553, Halsted, New York.

\bibitem[Sav]{savini:natural}
A.~Savini, \emph{Natural maps for measurable cocycles of compact hyperbolic
  manifolds}, preprint, https://arxiv.org/pdf/1909.07712.pdf.

\bibitem[Sav18]{savini2:articolo}
A.~Savini, \emph{Rigidity at infinity for lattices in rank-one lie groups}, to
  appear on J. Top. Anal. (2018).

\bibitem[Sto06]{St06}
P.~A. Storm, \emph{The minimal entropy conjecture for non uniform rank one
  lattices}, Geom. Func. Anal. \textbf{16} (2006), 959--980.

\bibitem[Thu79]{thurston}
W.~P. Thurston, \emph{The geometry and topology of $3$-manifolds}, Princeton,
  1979, mimeo\-graphed notes.

\bibitem[Zim80]{zimmer:annals}
R.~J. Zimmer, \emph{Strong rigidity for ergodic actions of semisimple lie
  groups}, Ann.~of Math. \textbf{112} (1980), no.~3, 511--529.

\bibitem[Zim82]{Zim82}
R.~J. Zimmer, \emph{Ergodic theory, semisimple lie groups, and foliations by
  manifolds of negative curvature}, Publications Math\'ematiques de l'Institut
  des Hautes \'Etudes Scientifiques \textbf{55} (1982), 37--62.

\end{thebibliography}

\end{document}